\documentclass[11pt]{amsart}

\usepackage{amsmath,amsfonts,amsthm,amssymb,graphicx,bm}
\usepackage{enumerate}
\usepackage[top=3cm, bottom=3cm, left=2.5cm, right=2.5cm]{geometry}
\usepackage{pdfsync,color}

\usepackage{bbm}

\usepackage[normalem]{ulem}

\newtheorem{theorem}{Theorem}

\newtheorem{proposition}[theorem]{Proposition}

\newtheorem{lemma}[theorem]{Lemma}
\newtheorem{remark}[theorem]{Remark}

\numberwithin{equation}{section}

\newcommand{\RR}{\mathbb{R}}

\newcommand{\eps}{\varepsilon}

\newcommand{\sign}{\mathrm{sign}\,}

\newcommand{\T}{\mathsf{T}}
\renewcommand{\L}{\mathsf{L}}
\newcommand{\Q}{\mathsf{Q}}
\renewcommand{\H}{\mathsf{H}}
\newcommand{\A}{\mathsf{A}}
\newcommand{\J}{\mathsf{J}}

\newcommand{\g}{g}
\newcommand{\ub }{u}
\newcommand{\bT}{\mathbf{T}}
\newcommand{\bB}{\mathbf{B}}

\newcommand{\R}{{\mathbb {R}}}

\newcommand{\scalar}[2]{\langle #1,#2\rangle}

    \makeatletter
    \let\@fnsymbol\@arabic
    \makeatother
    
\begin{document}

\title[Confinement by biased velocity jumps: aggregation of {\it Escherichia coli}]{Confinement by biased velocity jumps:\\aggregation of {\it Escherichia coli}}

\author[V. Calvez]{Vincent Calvez}
\address{Unit\'e de Math\'ematiques Pures et Appliqu\'ees, Ecole Normale Sup\'erieure de Lyon, CNRS UMR 5669, and project-team Inria NUMED, Lyon,
France.}
\email{\tt vincent.calvez@ens-lyon.fr}
\author[G. Raoul]{Ga\"el Raoul}
\address{Centre d'Ecologie Fonctionnelle et
Evolutive, CNRS UMR 5175, Montpellier, France.}
\email{\tt raoul@cefe.cnrs.fr}
\author[C. Schmeiser]{Christian Schmeiser}
\address{Fakult\"at f\"ur Mathematik, Universit\"at Wien,  Austria}
\email{\tt  christian.schmeiser@univie.ac.at} 			

\maketitle

\begin{abstract}
We investigate a linear kinetic equation derived from a velocity jump process modelling bacterial chemotaxis in the presence of an external chemical signal centered at the origin. We prove the existence of a positive equilibrium distribution with an exponential decay at infinity. We deduce a hypocoercivity result, namely: the solution of the Cauchy problem converges exponentially fast towards the stationary state. The strategy follows [J. Dolbeault, C. Mouhot, and C. Schmeiser, Hypocoercivity for linear kinetic equations conserving mass, Trans. AMS 2014]. The novelty here is that the equilibrium does not belong to the null spaces of the collision operator and of the transport operator. From a modelling viewpoint it is related to the observation that exponential confinement is generated by a spatially inhomogeneous bias in the velocity jump process.
\end{abstract}

\section{Introduction}

Unbiased velocity randomization by a jump process or by Brownian motion combined with acceleration by the force field produced by 
a confining potential can lead to convergence to an invariant probability measure, if the confinement is strong enough to balance
the dispersive effect of velocity randomization. For kinetic transport models of this kind, convergence to confined equilibria has been
studied extensively leading to strong convergence results with algebraic \cite{car,DesVil} and later also exponential \cite{DMS,HerNie} 
convergence rates. This is strongly related to the corresponding macroscopic description by Fokker-Planck equations of drift-diffusion
type \cite{AMTU}.

In this work a related type of particle dynamics is considered, where confinement is achieved by a biased velocity jump process, where
the bias replaces the confining acceleration field. The motivation comes from kinetic transport models for the chemotactic motility of 
microorganisms driven by gradients of chemo-attractors. The prototypical example is the bacterium {\it Escheria coli}, whose swimming 
pattern has been described as run-and-tumble \cite{BerBro,Berg-book}, meaning that periods of straight running alternate with periods of reorientation
(tumbling). Since typically tumble-periods are short compared to run-periods, models with instantaneous velocity jumps seem reasonable, but see \cite{Kafri-2008,Chatterjee-2011}. 
In the presence of a spatial chemo-attractant gradient, this stochastic process is biased upwards the gradient, although {\it E. coli} is too 
small to reliably measure the gradient along its length. An explanation for this phenomenon is that {\it E. coli} is able to measure gradients
in time along its path and increases its tumbling frequency, if it experiences decreasing chemo-attractant concentrations. This produces
the desired drift, even if the outward velocity after tumbling events is unbiased \cite{BerTur}.

This and similar motility behavior types have been incorporated into kinetic transport models of chemotaxis \cite{Alt1,OthDunAlt,Dolak-2005,Erban-Othmer-2005,Xue-Hwang-2010}, and a
connection to versions of the classical Patlak-Keller-Segel (PKS-) model \cite{KelSeg,Patlak} has been made by the macroscopic diffusion 
limit \cite{Alt2,HilOth-II}. If chemotaxis acts as a means of signalling between cells, it can be responsible for various types of pattern formation
with aggregation as the most important and basic outcome. 
Two examples of pattern formation are stable clusters of bacteria \cite{Mittal} and travelling pulses of bacterial colonies \cite{Adler66,Saragosti-PLoS,Saragosti}. For corresponding kinetic transport models with nonlinear coupling to a reaction-diffusion model for the chemo-attractant, the macroscopic diffusion limit produces nonlinear versions of the PKS-model \cite{CMPS,Hwang-Kang-Stevens-2005,Saragosti-PLoS}. We emphasize that hydrodynamic limits can also be derived from the kinetic model \cite{Dolak-2005,Filbet-Laurencot-Perthame-2005,James-Vauchelet-2013}. 

Aggregation patterns of signalling {\it E. coli} have been observed and simulated in \cite{Mittal}. The stochastic simulations have considered
a prescribed peaked stationary chemo-attractant concentration. A kinetic model corresponding to these simulations will be considered here:
\begin{equation}\label{kinetic}
  \partial_t f + v\cdot\nabla_x f = Q(f) = \int_V \left(K(x,v')f(t,x,v') - K(x,v)f(t,x,v)\right)dv' \,,
\end{equation}
where $f(t,x,v)$ is the phase (position($x$)-velocity($v$)-) space distribution of microorganisms at time $t\ge 0$, with $x\in\RR^d$ and
$v\in V\subset\R^d$. The velocity set $V$ is supposed to be bounded and rotationally symmetric. The right hand side of \eqref{kinetic}
is the {\em turning operator}. It describes the velocity changes due to tumbling. The turning kernel $K(x,v)$ is the rate of changing from
velocity $v$ to a different velocity $v'$ at position $x$. The $x$-dependence contains the influence of the chemo-attractant. The fact that 
the turning kernel only depends on the incoming (pre-tumbling) velocity means a complete randomization of velocity at tumbling events.
This is actually not describing the experimental evidence precisely. Whereas independence of the outgoing velocity distribution from the 
chemo-attractant gradient seems to be a valid assumption, some directional persistence of {\it E. coli} has been observed, {\em i.e.}  an outgoing 
velocity distribution biased by the incoming velocity \cite{BerBro}. This effect might have important quantitative consequences on the efficiency of chemotactic foraging strategies \cite{Schnitzer,Locsei,Nicolau},
but apparently does not change the qualitative picture.

The main deficiency of \eqref{kinetic} as a model for the experiments of \cite{Mittal} is the lack of a nonlinear coupling with an equation for
the chemo-attractant, describing production by the cells, diffusion, and decay. The restriction to the linear problem has purely mathematical 
reasons. The following sections will show that proving the existence and stability of aggregated stationary solutions already poses significant 
difficulties for a simple version of \eqref{kinetic}. Extensions to nonlinear models are the subject of ongoing investigations. Numerical 
studies \cite{Vauchelet-2010,Carrillo,Filbet-Yang} show that convergence to a steady  state can be expected under suitable assumptions.  We highlight the well-balanced numerical scheme proposed by Gosse in \cite[Section 9.4]{Gosse-book}. This focuses on the local kinetic equilibria of \eqref{kinetic} in a spatial finite volume with inflow boundary conditions, very much in the spirit of the present work.

We shall restrict ourselves to a one-dimensional model ($d=1$) with the velocity set $V = [-1/2,1/2]$ (chosen such that $|V|=1$).
A typical example for the choice of the turning kernel is given by
\begin{equation}\label{eq:example}
  K(x,v) = 1+\chi\sign(xv) \,,\qquad\mbox{with } 0<\chi<1\,.
\end{equation}
The turning rate takes the larger value $1+\chi$ for cells moving away from $x=0$, and the smaller value $1-\chi$ for cells moving
towards $x=0$. Equilibrium distributions of the turning operator, {\em i.e.} , multiples of $1/K$, have a jump at $x=0$ and are therefore
not solutions of \eqref{kinetic}. Stationary solutions have to balance the turning operator with the transport operator $v\partial_x$. Existence
and uniqueness (up to a multiplicative constant) of a stationary solution will be proven in the following section.
As an illustration we have a short look at the two-velocity model,  also known as the Cattaneo model for chemotaxis \cite{Hadeler-1997,HilOth-II,Hillen-2004,Hwang-Kang-Stevens-2006,Natalini-Ribot,Gosse-2012,James-Vauchelet-2013},
\begin{equation*} \begin{array}{l}
  \partial_t f^+ + \partial_x f^+ = ( 1 - \chi\sign(x) ) f^- - (1 + \chi\sign(x)) f^+ \,,\medskip\\
  \partial_t f^- - \partial_x f^- = ( 1 + \chi\sign(x) ) f^+ - (1 - \chi\sign(x)) f^- \,.
\end{array}
\end{equation*}
Here, steady states are easily found explicitly as multiples of $g^+(x) = g^-(x) = e^{-2\chi|x|}$. The exponential decay with respect to position
carries over to the model with $V=[-1/2,1/2]$, which is also proven in the following section. We also refer to \cite{Keener-Newby} for the analysis of a discrete velocity jump process and its stationary distribution based on the WKB expansion.

Section 3 is concerned with the decay to equilibrium as $t\to\infty$, employing the modified entropy approach of \cite{DMS} for abstract 
hypocoercive operators \cite{Villani}. This approach is based on a decomposition of the generator of the dynamics into a symmetric 
negative semidefinite operator and a skew-symmetric operator, where the latter needs to satisfy a certain coercivity condition on the 
null space of the former (called 'macroscopic coercivity' in \cite{DMS}, compare also to 'instability of the hydrodynamic description' in 
\cite{DesVil2}), and the former needs to be coercive on the complement of its null space ('microscopic coercivity' in \cite{DMS}). The set of equilibria has to be the intersection of the null spaces of both operators. In the present case, this does not 
permit the splitting into turning and transport operators, as usual for kinetic equations. The main preparatory step is therefore the definition 
of an appropriate splitting, before the modified entropy method is applied.

The terminology 'microscopic' and 'macroscopic' refers to an asymptotic limit, where the characteristic time scale of the symmetric operator
is assumed to be much smaller than that of the skew-symmetric operator. This limit is carried out in the first part of Section 4.
Typically the separation of time scales can be achieved by a macroscopic re-scaling of length in kinetic equations, which is not the case
in the present situation by the redefinition of the splitting between operators. The asymptotic limit has therefore not much biological relevance.
A second, more realistic macroscopic limit is carried out in the second part of Section 4, where smallness of the parameter $\chi$ is assumed,
and length and time is rescaled diffusively. Both macroscopic limits produce drift-diffusion equations, whose diffusivities and convection 
velocities are different, but with the same qualitative behavior. This brings us back to the beginning of the introduction, since it shows that
the macroscopic behavior created by biased velocity jumps is the same as for unbiased jumps combined with a confining potential.

\section{Existence and exponential decay of stationary solutions}

We seek a nonnegative stationary state $\g\in L_+^1(\RR\times V)$ satisfying
\begin{equation}
v\partial_x \g(x,v) + K(x,v) \g(x,v) - \int_{V} K(x,v') \g(x,v')\, dv'  = 0\, .\label{eq:equilibrium}
\end{equation}
For simplicity, we assume $V = [-1/2,1/2]$, such that $|V| = 1$.
More generally we may assume that we are given a probability measure $d\nu (v)$ which is compactly supported. Then we replace $dv$ with $d\nu(v) $ in the following. 
We denote $V_+ = V\cap [0,+\infty)$ and $V_- = V\cap (-\infty,0]$. 

We make the following assumptions on the turning kernel $K(x,v)$: 
\begin{enumerate}[{\bf (H1)}]
\item there exists $K_{\min}$ and $K_{\max}$ such that $ 0< K_{\min}\leq  K(x,v)\leq K_{\max}$,
\item $ K(x,v) =  K_+(v) $ for $x>0$ and $K(x,v) =  K_-(v) $ for $x<0$,
\item $\int_V v ( K_+(v))^{-1}\, dv < 0$,
\item $K$ is symmetric with respect to $x=0$, {\em i.e.}  $K_+(v) =  K_-(-v)$, and piecewise continuous, with a possible jump only at $v = 0$. 
\end{enumerate}
These conditions are satisfied by our main example \eqref{eq:example}.

\begin{theorem}\label{eq:existence st st}
Under the assumptions \textbf{(H1--H4)} there exists a  nontrivial stationary state $\g(x,v)$ solution to \eqref{eq:equilibrium}. It is positive, bounded and symmetric: $\g(x,v) = \g(-x,-v)$. There exists $\alpha>0$, a positive velocity profile $G$, and a constant $C>0$ such that 
\begin{equation}\label{g-est}
  \frac1C\, e^{-\alpha x} G(v) \leq \g(x,v) \leq C\, e^{-\alpha x} G(v) \,,
   \qquad  x\geq 0\,,\  v\in V \, . 
\end{equation}
\end{theorem}

In addition, we can state precisely the asymptotic behaviour of the stationary state $g$ as $x\to +\infty$, namely variable separates in the limit: $g(x,v) \sim H(g)e^{-\alpha x} G(v)$, where $H(g)$ denotes a constant. 
This is a specific feature of the so-called Milne problem in radiative transfer theory \cite{Bardos-Santos-Sentis,Golse-1987}. The analogy between the existence of a stationary state for equation \eqref{eq:equilibrium} and the Milne problem is the cornerstone of the proof of Theorem \ref{eq:existence st st}.  

\begin{proposition}\label{cor:milne}
The stationary state constructed in Theorem \ref{eq:existence st st} converges exponentially fast towards a multiple of the asymptotic profile $e^{-\alpha x} G(v)$ as $x\to +\infty$. More precisely, there exists a constant $C_0$ such that the following estimate holds true,
\begin{equation}\label{eq:exponential decay}
(\forall x\geq 0)\quad\left\| \dfrac{g(x,v)}{e^{-\alpha x}G(v)} - H(g) \right\|^2_{L^2(V;v^2G^2\, dv)} \leq C_0 e^{-\beta x}\,, 
\end{equation} 
where $\beta$ is the positive root of $ \frac12 \beta^2 + \alpha \beta - 2 \kappa = 0$, $\kappa = \left(\inf_{v\in V} \frac{K_+(v)}{v^2 G(v)} \right)^2\left(\int_{V}  v^2 G^2 \,dv\right)$, and the constant $H(g)$ is given by the formula:
\begin{equation}
H(g) = \dfrac{\int_V g(0,v) v^2 G(v) \, dv}{\int_V   v^2 G(v)^2 \, dv}\, .
\end{equation}
\end{proposition}

\begin{remark}
The higher-dimensional case is left open. We refer to \cite{Filbet-Yang} for accurate numerical simulations of \eqref{kinetic} in the two-dimensional case, where the velocity set is the unit sphere $\mathbb{S}^1$. In this work the authors clearly observe convergence towards a spherically symmetric stationary state. 
\end{remark}

\begin{figure}
\begin{center}
\includegraphics[width = 0.48\linewidth]{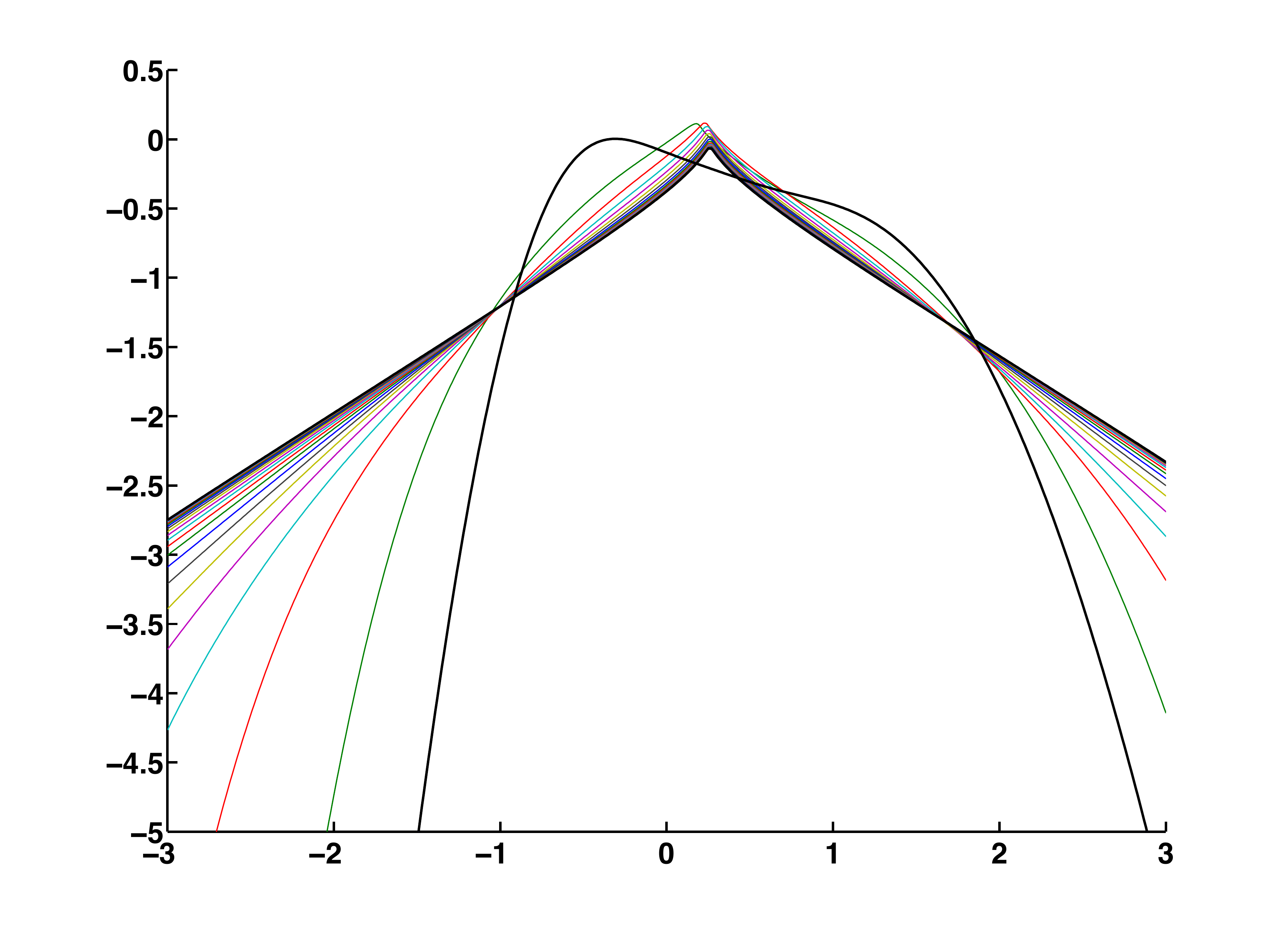}\;
\includegraphics[width = 0.48\linewidth]{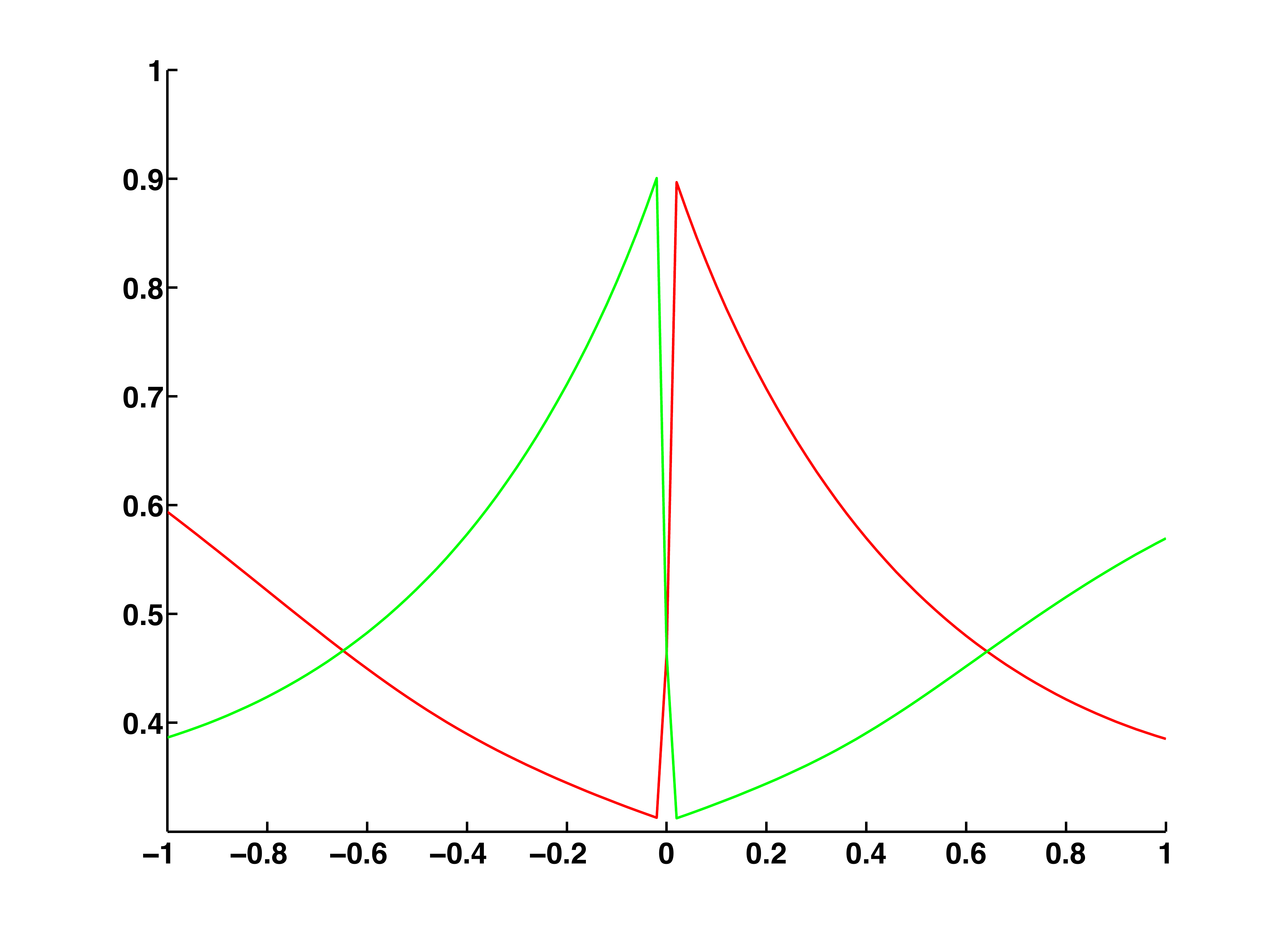}
\caption{\small Numerical simulation of \eqref{kinetic} for the tumbling kernel \eqref{eq:example}. The domain of computation is a box $[-L,L]\times V$ with specular reflection. (Left) The macroscopic density $\int_V f(t,x,v)\, dv$ is plotted in logarithmic scale for successive times. Notice the exponential decay in space of the final state. (Right) The renormalized velocity profiles at $x = \pm L$ (resp. red or green curve) are close to the expected profile $G(v)$. Note the discontinuity at $v = 0$. See also \cite[Section 9.4]{Gosse-book} for much more accurate simulations using a well-balanced scheme.}
\label{fig:simu}
\end{center}
\end{figure}

\begin{proof}[Proof of Theorem \ref{eq:existence st st}]
The proof is divided into several steps. The first step consists in deriving the correct asymptotic behavior of $\g$ as $|x|\to +\infty$. In the second step we make the link between our stationary problem and the so-called conservative Milne problem in the half-space \cite{Bardos-Santos-Sentis}. The core of the proof is a regularization process in the velocity variable at $x = 0$ (step 3), which enables us to apply the Krein-Rutman Theorem (step 4). 

The proof shares some similarity with analogous problems in homogenization theory (see for instance \cite{Bal,Goudon-Mellet}). However the connection with the Milne problem is new up to our knowledge. 

\paragraph{Step\#1. Exponential decay as $x\to +\infty$.}
We make the following ansatz
\[ \g(x,v)\sim e^{-\alpha x}  G(v)\quad \text{as}\; x\to +\infty\,.\]
Substituting this ansatz in \eqref{eq:equilibrium} yields
\begin{align}
&-\alpha v  G(v) +  K_+(v)  G(v) = \int_V  K_+(v')  G(v')\, dv'\, , 
\label{eq:dispersion1}
\\
&  G(v) = ( K_+(v) - \alpha v)^{-1} \int_V  K_+(v')  G(v')\, dv' \, . \nonumber
\end{align}
Clearly the profile $G$ is characterized up to a constant factor. We opt w.l.o.g. for the renormalization $
\int_V  K_+(v)  G(v)\, dv  = 1$.
Therefore the exponent $\alpha$ is characterized by the dispersion relation:
\begin{equation} \label{alpha-equ}
   \int_V  K_+(v)( K_+(v) - \alpha v)^{-1}\, dv = 1 \, . 
\end{equation}
We now prove that this defines a unique $\alpha > 0$ such that $(\forall v\in V)\;K_+(v) - \alpha v  > 0$. 
We introduce the auxiliary function $J(\alpha) = \int_V  K_+(v)( K_+(v) - \alpha v)^{-1}\, dv$. It satisfies
\begin{enumerate}[(i)]
\item $J(0) = 1$ and
\[ \lim_{\alpha \to \underset{v> 0}{\inf} \left(v^{-1}  K_+(v)\right) } J(\alpha) = \lim_{\alpha \to \underset{v< 0}{\sup} \left( v^{-1}  K_+(v)\right)} J(\alpha) = +\infty\, .\]
\item $J'(0) = \int_V v ( K_+(v))^{-1}\, dv < 0$ (the average speed is negative on the far right: this is the confinement effect).
\item The function $J$ is convex on the admissible range of $\alpha$.
\end{enumerate}
As a consequence there exists a unique $\alpha \in \left(0,\underset{v> 0}{\inf} \left(v^{-1}  K_+(v)\right)\right)$ such that $J(\alpha) = 1$.

By symmetry, a similar ansatz can be made on the left side: $\g(x,v)\sim   e^{\alpha x}  G(-v)$ as $x\to -\infty$.

\paragraph{Step\#2. Connection with the conservative Milne problem.}

For $x\geq 0$ we define $\ub$ by  $\g(x,v) = e^{-\alpha x} G(v) \ub(x,v)$. We deduce from \eqref{eq:equilibrium} that it satisfies the following equation,
\begin{equation}   v\partial_x \ub + \int_{V} G^{-1} K_+' G'\left( \ub - \ub' \right) \, dv'  = 0\, ,\quad (x,v)\in \R_+\times V\, . \label{eq:u} \end{equation}
The profile $\g(0,v) = G(v) \ub(0,v)$ is unknown, of course. On the other hand, its knowledge is sufficient to reconstruct the entire function $\ub$. This is the purpose of the Milne problem. It states that for a given profile $u(0,v)$ defined for $v>0$ only, there exists a unique bounded function $u$, defined over $\RR_+\times V$, solution of  \eqref{eq:u}. 

\begin{lemma}\label{lem:Milne}
Let $\varphi \in L^\infty(V_+)$. Then there exists a unique bounded solution $u(x,v)$ of the Milne problem
\begin{equation}\label{eq:BVP u}
\left.\begin{array}{ll}
\displaystyle v\partial_x u + \int_{V} G^{-1} K_+' G'\left( u - u' \right) \, dv'  = 0\, ,\quad & (x,v)\in \R_+\times V \,,\smallskip\\
u(0,v) = \varphi(v) \, , \quad & v\in V_+ \,,
\end{array}\right\}
\end{equation}
satisfying the pointwise estimate,
\begin{equation}\label{u-est}
\inf_{V_+}\varphi \le u(x,v)\le \sup_{V_+}\varphi \,,\qquad (x,v)\in \R^+\times V\,.
\end{equation}
\end{lemma}

\begin{proof}[Proof of Lemma \ref{lem:Milne}] This result is classical (see \cite{Bardos-Santos-Sentis} and references therein). We recall the main lines of the proof for the sake of completeness. 

For $\eps>0$ we associate the perturbed problem
\begin{equation}\label{eq:u_e}
\left.\begin{array}{ll}
\displaystyle \eps u_\eps +  v\partial_x u_\eps + \int_{V} G^{-1} K_+' G'\left( u_\eps - u_\eps' \right) \, dv'  = 0\, ,\quad & (x,v)\in \R_+\times V \,,\smallskip\\
u_\eps(0,v) =  \varphi(v) \, , \quad & v\in V_+ \,,
\end{array}\right\}
\end{equation}
which possesses a unique solution in $L^\infty(\R_+\times V)$, as can be proven using a fixed point argument as explained below. The Duhamel formulation for \eqref{eq:u_e} reads
\begin{equation}\label{eq:Duhamel} u_\eps(x,v) = \begin{cases} 
\displaystyle \int_0^{x/v} G(v)^{-1}\exp\left(-\lambda_\eps(v)s \right)A_\eps(x - sv)\, ds +\exp\left(-\lambda_\eps(v)x/v\right) \varphi(v) \,, & v>0 \,,\\
\displaystyle \int_0^{+\infty}G(v)^{-1}\exp\left(-\lambda_\eps(v)s \right) A_\eps(x - sv)\, ds \,, & v<0 \,,
\end{cases} \end{equation}
where $\lambda_\eps(v) = G(v)^{-1}+\eps$, and the macroscopic quantity is defined by $A_\eps(x) = \int_{V} K_+(v') G(v') u_\eps(x,v')\, dv'$.
For a given inflow data $\varphi$, we associate the map $\bT_\eps$ from $L^\infty(\R_+)$ into itself,  
\[
A_\eps \longmapsto  \int_{V} K_+(v') G(v') u_\eps(\cdot,v')\, dv'\,, 
\]
where $u_\eps$ is defined by \eqref{eq:Duhamel}. Then $\bT_\eps$ is a contraction with rate
\[ \sup_{v\in V} \dfrac{1}{1+\eps G(v)} < 1\, . \]
The fixed point of this map is a solution of \eqref{eq:u_e}. 

The maximum principle applied to \eqref{eq:u_e} implies that $u_\eps$ satisfies \eqref{u-est}. Therefore, as $\eps\to 0$, we can extract 
a subsequence that converges  in $L^\infty(\R_+)$-weak$^*$ to a solution $u$ of \eqref{eq:BVP u}, which satisfies 
\begin{equation}
\|u\|_{L^\infty(\R_+\times V)} \leq \|\varphi\|_{L^\infty( V_+)}\,.
\label{eq:linfty bound}
\end{equation}
Next we show that any bounded solution of \eqref{eq:u} satisfies the estimate \eqref{eq:linfty bound}, implying uniqueness.  
We define $U(x,v) = \frac{G(v)}{K(v)}e^{\alpha x}$, which also satisfies the differential equation in \eqref{eq:BVP u}. It corresponds in fact to the function $g = \frac1K$, which is a trivial solution of \eqref{eq:equilibrium} for $x>0$.
For $\eps>0$ we consider $w = u - \eps U$. It satisfies \eqref{eq:u} with inflow data $\varphi_\eps = \varphi - \eps U(0,\cdot)$. Since $u$ is bounded and  $- \eps U \to -\infty$ as $x\to+\infty$, it is clear that $w$ has a maximum in $\R_+\times V$. Notice that $w$ is not necessarily continuous at $v = 0$, but piecewise continuity is sufficient. The maximum cannot be attained in $(0,\infty)\times V$ by the maximum principle. Therefore it is attained at $x=0$ and we get 
$$
 u(x,v)-\eps U(x,v)  \leq \sup_{V_+} (\varphi - \eps U(0,\cdot)) \,,\qquad (x,v)\in \R_+\times V\,.
$$ 
A similar estimate holds for $ -u - \eps U$. Letting $\eps\to 0$ we deduce that $u$ satisfies \eqref{u-est}, and therefore also \eqref{eq:linfty bound}. 
\end{proof}

\paragraph{Step\#3. Compactness.}

The expected symmetry property $\g(0,v) = \g(0,-v)$ of the equilibrium distribution motivates the definition of  the fixed point operator
$\bB:\, L^\infty(V_+) \to L^\infty(V_-)$:
\begin{equation}\label{eq:in-out}
(\bB\varphi)(v) = \dfrac{G(v)}{G(-v)} ({\mathcal A}\varphi)(v) \,,\qquad v\in V_- \,,
\end{equation}
with the Albedo operator $({\mathcal A}\varphi)(v)=u(0,v)$, where $u$ is the unique bounded solution of \eqref{eq:BVP u}.

\begin{lemma}\label{lem:compactness}
The operator $\bB$ is compact and positive. 
\end{lemma}
\begin{proof}
To prove compactness, we first define the macroscopic quantity as above,
\[ A (x) = \int_{V} K_+(v') G(v') u (x,v')\, dv'\, . \]
We have $u\in L^\infty(\R_+\times V)$ and  from \eqref{eq:BVP u} also $v\partial_x u\in L^\infty(\R_+\times V)$. The one-dimensional averaging 
lemma \cite{Golse-Lions-Perthame-Sentis} implies $A\in \mathcal C^{0,\theta}(\R_+)$ for all $\theta\in (0,1)$, with the corresponding H\"older 
semi-norm bounded in terms of $\|\varphi\|_{L^\infty}$. We deduce from the Duhamel representation formula
$$
  u(0,v) = \int_0^{+\infty}G(v)^{-1}\exp\left(-G(v)^{-1}s  \right) A( - sv)\, ds
 = \int_0^{+\infty} e^{- t} A( -  t vG(v))\, dt\, ,\qquad v\in V_- \,,
$$
that $u(0,v)$ is uniformly continuous for $v<0$, with a modulus of continuity which depends only on $\|\varphi\|_{L^\infty}$ and the modulus of continuity of $vG(v)$ on $V_-$.  

Positivity is immediate from the Duhamel formula since $A(x)>0$ for $x>0$, as soon as $\varphi\geq 0$ and $\varphi\neq 0$. 
\end{proof}

\paragraph{Step\#4. Conclusion.}

In order to apply the Krein-Rutman Theorem \cite{Dautray-Lions}, we consider the restriction of $\bB$ to continuous functions on $V_+$, since the interior of $L^\infty(V_+)$ is empty. Lemma \ref{lem:compactness} also holds for $\bB|_{\mathcal C^0(V_+)}$. 

The Krein-Rutman theorem states that the operator $\bB|_{\mathcal C^0(V_+)}$ possesses a simple dominant eigenvalue $\lambda\in \R$ together with a positive eigenfunction $\varphi$: $\bB \varphi = \lambda \varphi$. The conservation property for \eqref{eq:equilibrium} yields $\lambda = 1$ by
the following argument: We denote by $\ub\in L^\infty(\R_+\times V)$ the solution of the Milne problem \eqref{eq:BVP u} with inflow data 
$\ub(0,v) = \varphi(v)$. We define accordingly $\g(x,v) = e^{-\alpha x} G(v) \ub(x,v)$. It is a solution of  \eqref{eq:equilibrium} on $\R_+\times V$ satisfying
\[ 
  \lambda \g(0,v) = \g(0,-v)\,,\qquad v\in V_+\, . 
\]
Integrating \eqref{eq:equilibrium} over $\R_+\times V$ yields
\begin{equation*}
0 = \int_V v\g(0,v)\, dv = (1 - \lambda) \int_{V_+} v\g(0,v)\, dv\,,
\end{equation*}
implying $\lambda = 1$ by the positivity of $\g$.  It is straightforward to check that $\g$, symmetrically extended to $\R\times V$, is a solution of 
\eqref{eq:equilibrium}, satisfying \eqref{g-est} as a consequence of \eqref{u-est}. 
\end{proof}

\begin{proof}[Proof of Proposition \ref{cor:milne}]
We adapt the method of \cite{Bardos-Santos-Sentis}. We shall use a quantitative energy/energy dissipation approach with respect to the space variable. First, we integrate \eqref{eq:u} against $K_+ G^2$, in order to derive a non trivial conservation,
\begin{equation}\label{eq:conservation 1}  
\partial_x \left(\int_V v  K_+(v)  G(v)^2 u(x,v) \, dv  \right) = 0\, . \end{equation}
This comes in addition to the zero-flux relation 
\begin{equation}\label{eq:conservation 2}  
\int_V v G(v) u(x,v)\, dv = 0\, ,  \end{equation}
which is a straightforward consequence of equation \eqref{eq:equilibrium} after integration with respect to the velocity variable. We observe that the relation \eqref{eq:conservation 1} combined with \eqref{eq:conservation 2} and \eqref{eq:dispersion1} yields 
\begin{align*}
cste = \int_V v  K_+(v)  G(v)^2 u(x,v) \, dv  & = \int_V v \left(  \alpha v G(v) + \int_V  K_+(v')  G(v')\, dv' \right) G(v)  u(x,v) \, dv \\
& = \alpha \int_V v^2 G(v)^2 u(x,v) \, dv\, .
\end{align*}

Secondly, we multiply \eqref{eq:u} by $2K_+ G^2   (u - H(u))$, where the constant $H(u)$ is defined such that 
\begin{equation}\label{eq:conservation 3} \int_V v^2   G(v)^2 u(x,v) \, dv  =  H(u) \int_V v^2   G(v)^2  \, dv \, . \end{equation}
We obtain 
\begin{align}
\partial_x \left(\int_V v  K_+ G^2 \left( u - H(u)  \right)^2  \, dv  \right) & = - \iint_{V\times V}   K_+  G  K_+'  G' \left( u - u'\right)^2  \, dv' dv \nonumber
\\
& \leq -   2 \left(\inf_{v\in V} \frac{K_+(v)}{v^2 G(v)} \right)^2 \left(\int_{V}  v^2 G^2 \,dv\right) \int_{V}  v^2 G ^2 \left( u - H(u)  \right)^2  \, dv \, , \label{eq:dissipation}
\end{align}
where we have expanded $\left( u' - u\right)^2 = \left( u' - H(u)\right)^2 + \left( u - H(u)\right)^2 + 2\left(H(u) - u\right)\left( u' - H(u)\right)$, and we have used the conservation \eqref{eq:conservation 3}. We define two auxiliary quantities,
\[ J(x) =\int_V v  K_+(v) G(v)^2 \left( u(x,v) - H(u)  \right)^2  \, dv \, , \quad E(x) = \int_{V}  v^2 G(v) ^2 \left( u(x,v) - H(u)  \right)^2  \, dv \, . \]
The dissipation inequality \eqref{eq:dissipation} reads $\partial_x J(x) + 2 \kappa E(x) \leq 0$. On the other hand, multiplying  \eqref{eq:u} by $2vG^2 (u - H(u))$, we get,
\begin{align*}
\partial_x\left( \int_V v^2 G^2 (u - H(u))^2 \, dv \right) & = - 2 \iint_{V\times V} v G K_+' G' (u - H(u))^2 \, dv' dv \\Ê
& \quad + 2 \iint _{V\times V} v G K_+' G' (u' - H(u))(u - H(u)) \, dv' dv \\
& =  2 \int_V vG \left( \alpha v G - K_+G \right)(u - H(u))^2 \,  dv\, ,
\end{align*}
where we have used \eqref{eq:dispersion1}, \eqref{eq:conservation 2} and $\int_V vG\, dv = 0$. This reads also $\partial_x E(x) = 2 \alpha E(x) - 2 J(x)$. All together this yields the damped second order inequality
\[ - \frac12 \partial^2_x E(x) + \alpha \partial_x E(x)  + 2 \kappa E(x) \leq 0\, . \]
Complemented with the information that $E$ is bounded, this is sufficient to prove that $E$ decays exponentially fast. In fact, let $\beta>0$ be the positive root of $ \frac12 \beta^2 + \alpha \beta - 2 \kappa = 0$, the function $\tilde E(x) = e^{\beta x} E(x)$ satisfies
\[ -\frac12 \partial^2_x \tilde E(x) + (\alpha + \beta) \partial_x \tilde E(x)\leq 0\, . \]
Therefore, there exists a constant $C_0$, depending only on $E(0)$ and $\partial_x E(0)$, such that 
\[ -  \partial_x \tilde E(x) + 2 (\alpha + \beta)   \tilde E(x) \leq C_0\, . \]
We deduce that for all $0\leq x<y$, we have
\[ \tilde E(x) e^{-2(\alpha + \beta) x} - \tilde E(y) e^{-2(\alpha + \beta) y} \leq \frac{C_0}{2(\alpha + \beta)} \left( e^{-2(\alpha + \beta)x} -  e^{-2(\alpha + \beta)y} \right)\, . \]
Hence we conclude that,
\[ E(x) \leq E(y) e^{-(2\alpha + \beta)(y-x)} + \frac{C_0}{2(\alpha + \beta)}   e^{- \beta x}\, . \]
Since $E(y)$ is uniformly bounded, letting $y\to +\infty$ this yields that $E(x)$ decays exponentially fast, {\em i.e.} the estimate \eqref{eq:exponential decay} holds true up to a modification of the constant $C_0$, and the abuse of notation $H(g) = H(u)$.
\end{proof}

\section{Decay to equilibrium by hypocoercivity}

Since \eqref{kinetic} conserves the total mass, we expect $\lim_{t\to\infty} f(\cdot,\cdot,t) = \mu_\infty g$ (with $g$ solving 
\eqref{eq:equilibrium}), where the constant $\mu_\infty$
is chosen such that $\mu_\infty \iint_{\RR\times V} g(x,v)dv\,dx= \iint_{\RR\times V} f(x,v,0)dv\,dx$. Replacing $f$ by $f-\mu_\infty g$, 
we may assume $\mu_\infty = 0$, and therefore 
$$
  \iint_{\RR\times V} f(x,v,t)dv\,dx= 0  \,,\qquad\mbox{for } t\ge 0 \,, 
$$
in the following.

The decay to equilibrium will be proved by employing the abstract procedure of \cite{DMS}. It is based on a situation, where the equilibrium
lies in the intersection of the null spaces of the collision and the transport operator. However, the equilibrium $g$ 
is not in the null spaces of the collision operator $\Q f := \int_V K'f'dv' -Kf$ and of the transport operator $v\partial_x$. Therefore, as a first step,
collision and transport operators will be redefined.

Multiplying \eqref{kinetic} by $f/g$, we get:
\begin{equation*}
\dfrac{1}{2}\,\dfrac d{dt} \|f\|^2 = -  \frac12\iiint_{\RR\times V^2} g' K' \left( \dfrac{f'}{g'} - \dfrac{f}{g}\right)^2 dv'dv\,dx \,,
\end{equation*}
where $\|\cdot\|$ is the norm on 
\begin{equation}\label{def:H}
  \mathcal{H} := \left\{ f\in L^2\left(\frac{dv\,dx}{g}\right):\, \int_{\RR\times V} f\,dv\,dx=0 \right\} \,,
\end{equation} 
induced by the scalar product
\begin{equation*}
  \langle f_1,f_2\rangle := \iint_{\RR\times V} \dfrac{f_1 f_2}{g} dv\,dx \,.
\end{equation*}
This motivates the definition of the symmetrized collision operator
\begin{equation} \label{def-L}
 \L f := \int_V \dfrac{g' K' + gK}{2}\left( \dfrac{f'}{g'} - \dfrac f g\right)\, dv' \,,
\end{equation}
with the same dissipation:
\begin{equation*}
  \langle \L f_1,f_2\rangle = \langle f_1,\L f_2\rangle \,,\qquad  \dfrac{1}{2}\,\dfrac d{dt} \|f\|^2 = \langle \L f,f \rangle \,.
\end{equation*}
Hence we rewrite \eqref{kinetic} as $\partial_t f + \T f = \L f$ with the modified transport operator
\begin{equation} \label{def-T}
  \T f := v\partial_x f + \dfrac{1}{2} \int_V\left(Kf-K'f' + Kg\dfrac{f'}{g'} - K'g' \dfrac fg \right)dv'\,.
\end{equation}
It is easily checked that $\T g = \L g = 0$ and that $\T$ is skew symmetric with respect to $\langle\cdot,\cdot\rangle$. 
Obviously, for fixed $x$, the null space of $\L$ is spanned by $g(x,\cdot)$. The orthogonal projection to $\mathcal{N}(\L)$ is given by
\begin{equation*}
  \Pi f := \rho_f \dfrac{g}{\rho_g} \,,\qquad\mbox{with } \rho_f := \int_V f\,dv \,.
\end{equation*}
We also observe
\begin{equation} \label{rhoTf}
  \rho_{\T f} = \partial_x \int_V v f\, dv\,, 
\end{equation}
implying $\Pi\T\Pi=0$ (by the property $\int_V vg\,dv = 0$ of the equilibrium distribution), which is Assumption (H3) in the abstract
setting of \cite{DMS}, Section 1.3. The so called 'microscopic coercivity' Assumption (H1) is the subject of the following result.

\begin{lemma} \label{lem:micro-coercivity}
With the above definitions, $-\langle \L f, f\rangle \ge K_{min}\|(1-\Pi)f\|^2$ holds with $K_{min}=\inf_{\RR\times V} K$ for every 
$f\in \mathcal{H}$.
\end{lemma}

\begin{proof}
\begin{eqnarray*}
  -\langle \L f, f\rangle &\ge& \dfrac{K_{min}}{2} \iiint_{\RR\times V^2} (g+g') \left( \frac{(1-\Pi)f'}{g'} - \frac{(1-\Pi)f}{g}\right)^2 dv'dv\,dx \\
  &=& \dfrac{K_{min}}{2} \iiint_{\RR\times V^2} (g+g') \left( \left(\frac{(1-\Pi)f'}{g'}\right)^2 + \left(\frac{(1-\Pi)f}{g}\right)^2 \right)dv'dv\,dx \\
  &\ge& \dfrac{K_{min}}{2} \iiint_{\RR\times V^2}  \left( \frac{(1-\Pi)f'^2}{g'} + \frac{(1-\Pi)f^2}{g} \right)dv'dv\,dx = K_{min} \|(1-\Pi)f\|^2 \,.\\
\end{eqnarray*}
\end{proof}

The next step is 'macroscopic coercivity' (Assumption (H2) in \cite{DMS}). It relies on the asymptotic behavior of $g$ as $|x|\to\infty$.
By Theorem 1, the equilibrium distribution satisfies
\begin{equation} \label{g-ass1}
   0 < u_{min} e^{-\alpha|x|} \le g(x,v) \le u_{max} e^{-\alpha |x|} \,,\qquad \mbox{for } (x,v)\in \RR\times V \,,
\end{equation}
where $\alpha$ is the unique positive solution of the dispersion relation \eqref{alpha-equ}, and $u_{min}, u_{max}$ are positive
constants.

\begin{lemma} \label{lem:macro-coercivity}
Let \eqref{g-ass1} hold. Then, with the above definitions, there exists a constant $\lambda_M>0$, such that 
$\|\T\Pi f\|^2 \ge \lambda_M \|\Pi f\|^2$ holds for all $f\in \mathcal{H} \cap \mathcal{D}(\T\Pi)$.
\end{lemma}

\begin{proof}
A straightforward computation gives
$$
  \|\T\Pi f\|^2 = \int_{\RR} \left( \partial_x \dfrac{\rho_f}{\rho_g} \right)^2 m_g dx \,,\qquad\mbox{with } m_g(x) 
  =\int_V v^2 g(x,v)dv \,.
$$
By the boundedness of the velocity space, $\rho_g$ and $m_g$ satisfy estimates of the form \eqref{g-ass1}, implying the existence of
a positive constant $c$ such that
$$
  \|\T\Pi f\|^2 \ge c \int_{\RR} \left( \partial_x \dfrac{\rho_f}{\rho_g} \right)^2 \rho_g dx \,.
$$
We claim that the measure $\rho_g dx$ permits a Poincar\'e inequality. This can be proved via bounds on the spectrum of Schr\"odinger 
operators. It is a consequence of the fact that $\rho_g$ can be bounded from above and below by multiples of a smooth function, whose 
logarithm is asymptotically linear as $|x|\to\infty$ and Theorem A.1 of \cite{Villani}. Therefore, there exists a positive constant $\lambda_M$
such that
$$
  \|\T\Pi f\|^2 \ge \lambda_M \int_{\RR} \left( \dfrac{\rho_f}{\rho_g} \right)^2 \rho_g dx = \lambda_M \|\Pi f\|^2 \,.
$$
\end{proof}

The approach of \cite{DMS} relies on the 'modified entropy'
$$
  \H[f] := \dfrac{1}{2} \|f\|^2 + \varepsilon \langle \A f,f\rangle \,,\qquad\mbox{with } \A := (1 + (\T\Pi)^*\T\Pi)^{-1} (\T\Pi)^* \,,
$$
and with a small positive constant $\varepsilon$. Its time derivative is given by
$$
\frac{d}{dt}\H[f] = \scalar{\L\,f}{f} - \eps\,\scalar{\A\,\T\,\Pi\,f}f - \eps\,\scalar{\A\,\T\,(1-\Pi)\,f}f + \eps\,\scalar{\T\,\A\,f}f 
  + \eps\,\scalar{\A\,\L\,f}{f} \,.
$$
The first two terms on the right hand side are responsible for the decay of $\H[f]$, noting microscopic coercivity 
(Lemma \ref{lem:micro-coercivity}) and
the fact that the macroscopic coercivity result of Lemma \ref{lem:macro-coercivity} implies 
$$
  \scalar{\A\,\T\,\Pi\,f}f \ge \frac{\lambda_M}{1+\lambda_M} \|\Pi f\|^2 \,.
$$
Coercivity of the dissipation of $\H[f]$ can be obtained by choosing $\eps$ small enough, if the auxiliary operators 
$\A\T(1-\Pi)$, $\T\A$, and  $\A\L$ are bounded, since they only act on $(1-\Pi)f$. By Lemma 1 of \cite{DMS}, $\A$ and $\T\A$
are bounded. 

\begin{lemma} \label{lem:L-bounded}
Let \eqref{g-ass1} hold. Then the operator $\L$, defined in \eqref{def-L}, is bounded.
\end{lemma}

\begin{proof}
The result follows from the estimate
\begin{eqnarray*}
  |\scalar{\L f}f| &\le& 2 \iiint_{\RR\times V^2} g'K' \left( \frac{{f'}^2}{{g'}^2} + \frac{f^2}{g^2} \right) dv'dv\,dx
  \le 2K_{max} \left( \|f\|^2 + \iint_{\RR\times V} \frac{\rho_g}{g}\,\frac{f^2}{g} dv\,dx \right) \\
  &\le& 2 K_{max} \left( 1 + \frac{u_{max}}{u_{min}} \right) \|f\|^2 \,.
\end{eqnarray*}
\end{proof}

It remains to prove boundedness of $\A\T(1-\Pi)$, which is equivalent to an elliptic regularity estimate.

\begin{lemma}\label{lem:ell-reg}
Let \eqref{g-ass1} hold. Then the operator $\A\T(1-\Pi)$ is bounded.
\end{lemma}

\begin{proof}
As in \cite{DMS} we work on the adjoint. It is sufficient to prove boundedness of
$$
  (\A\T)^* = -\T^2\Pi (1+(\T\Pi)^*\T\Pi)^{-1} \,.
$$
Introducing $\varphi = (1+(\T\Pi)^*\T\Pi)^{-1} f$, the scalar product of the equivalent equation
\begin{equation} \label{phi-equ}
  \varphi + (\T\Pi)^* \T\Pi\varphi = f
\end{equation}
with $\varphi$ leads to the estimate $\|\varphi\|, \|\T\Pi\varphi\| \le \|f\|$.
Application of $\Pi$ to \eqref{phi-equ} gives
\begin{equation} \label{eq:rho_phi}
  \rho_g u_{\varphi} - \partial_x \left( m_g \partial_x u_{\varphi} \right) = \rho_f \,,
\end{equation}
with $\rho_{\varphi} = \rho_g u_{\varphi}$. We shall have to estimate the norm of 
$$
  \T^2\Pi\varphi =  g \,v^2 \partial_x^2 u_{\varphi} + \frac{1}{2}\partial_x u_{\varphi}  \int_V (v'-v)(Kg - K'g')dv' \,,
$$
satisfying
$$
  |\T^2\Pi\varphi| \le C e^{-\alpha|x|} \left( \left|\partial_x^2 u_{\varphi}\right| 
   + \left|\partial_x u_{\varphi} \right| \right) \,,
$$
in terms of $\|f\|$, which is a weighted $L^2\to H^2$ regularisation result for \eqref{eq:rho_phi}. The first order
derivative has already been taken care of by the bound for $\T\Pi\varphi = gv \partial_x u_{\varphi}$.
We shall also need
$$
  |\partial_x m_g| = \left| \iint_{V^2} v(Kg-K'g')dv'dv \right| \le C e^{-\alpha|x|} \,,
$$
where we have used the equation \eqref{eq:equilibrium}, satisfied by $g$, and \eqref{g-ass1}. Now \eqref{eq:rho_phi}
is rewritten as 
$$
  m_g \partial_x^2 u_{\varphi} = \rho_{\varphi} 
  - \partial_x m_g\partial_x u_{\varphi} - \rho_f \,,
$$
and the proof is completed by taking $L^2$-norms with the weight $e^{\alpha|x|}$, noting that $m_g\ge ce^{-\alpha|x|}$.
\end{proof}

We have completed the verification of the assumptions of Theorem 2 of \cite{DMS} and arrive at our main result:

\begin{theorem}
Let a stationary positive solution $g$ of \eqref{kinetic} (unique up to a constant factor) satisfy \eqref{g-ass1}, let $f_I\in L^2(dv\,dx/g)$
($\subset L^1(dv\,dx)$), 
and let
$$
  f_\infty(x,v) := g(x,v) \int_{\RR\times V} f_I\, dv\,dx \left( \int_{\RR\times V} g\, dv\,dx\right)^{-1} \,.
$$
Then the solution of \eqref{kinetic} subject to $f(t=0)=f_I$ satisfies
$$
   \|f(t,\cdot,\cdot)-f_\infty\|_{L^2(dv\,dx/g)} \le C e^{-\lambda t} \|f_I-f_\infty\|_{L^2(dv\,dx/g)}\,,
$$
with positive constants $C$ and $\lambda$, only depending on $\chi\in(0,1)$.
\end{theorem}

\section{Macroscopic limits}

\subsection*{The 'macroscopic limit' corresponding to the modified entropy method}

The separation into microscopic and macroscopic contributions employed in the previous section can be motivated by a
macroscopic limit, based on the assumption of a separation of time scales related to the collision and the transport operators.
With an appropriate (parabolic) rescaling of time, this leads to studying the limit as $\eps\to 0+$ in 
\begin{equation} \label{macro-scaling}
  \eps \partial_t f^\eps + \T f^\eps = \dfrac 1\eps \L f ^\eps \,.
\end{equation}
Whereas for the standard transport operator $v\partial_x$ the scale separation can be achieved by a length rescaling, the 
introduction of the 'Knudsen number' $\eps$ is completely artificial in the present situation, since the modified collision and transport
operators contain identical terms whose different weighting cannot be justified by scaling arguments. This is the reason for the
quotation marks in the title of this subsection.

The limit $\eps\to 0$ in the abstract equation \eqref{macro-scaling} has been carried out formally in \cite{DMS} under the assumptions 
$\Pi\T\Pi=0$, already used above, and that the restriction of the collision operator $\L$ to $(1-\Pi)\mathcal{H}$ possesses an inverse $\J$.
The formal limits $f^0$ of solutions of \eqref{macro-scaling} satisfy
\begin{equation}\label{macro-limit}
  f^0 \in \Pi\mathcal{H} \,,\qquad \partial_t f^0 = (\T\Pi)^* \J (\T\Pi) f^0 \,.
\end{equation}
Note that the macroscopic coercivity estimate in Lemma \ref{lem:macro-coercivity} is related to the simplified version $-(\T\Pi)^* (\T\Pi)$
of the macroscopic evolution operator. Since $f^0(x,v,t) = \rho^0(x,t) g(x,v)/\rho_g(x)$, the above evolution equation is equivalent to an 
equation for $\rho^0$.

\begin{lemma} \label{lem-Lf=h}
Let $\L$ be defined by \eqref{def-L}. Then the equation $\L f = h$ is solvable for $f$, iff $h\in (1-\Pi)\mathcal{H}$. For every solution there 
exists a function $\mu(x)$ such that
\[
  f = \mu g + \frac{2g}{\lambda + gK} \left( \frac{\int_V h/(\lambda+gK)dv}{\int_V 1/(\lambda+gk)dv} - h\right) \,,\qquad\mbox{with }
  \lambda(x) = \int_V g(x,v)K(x,v)dv \,.
\]
The additional requirement $f\in (1-\Pi)\mathcal{H}$ determines $\mu$ uniquely.
\end{lemma}

\begin{proof}
With $\nu = \int_V Kf\,dv$, $\mu = \int_V f/g\,dv$, the equation $\L f = h$ can be rewritten as
$$
  f = g\left( \mu + \frac{\nu-\lambda\mu-2h}{\lambda+gK} \right)
$$
Division by $g$ and integration with respect to $v$ gives an equation for $\nu-\lambda\mu$ leading to the claimed result. A second
equation for $\mu$ and $\nu$ can be obtained by multiplication by $K$ and integration. The coefficient matrix of the resulting system is 
non-invertible, and the solvability condition turns out to be $\int_V h\,dv = 0$, {\em i.e.}  $h\in (1-\Pi)\mathcal{H}$.
\end{proof}

\begin{proposition}
Let $\L$ and $\T$ be defined by \eqref{def-L} and, respectively, \eqref{def-T}. Then the equation \eqref{macro-limit} is equivalent to
\begin{equation}\partial_t \rho^0 = \partial_x\left(D\,\partial_x\left(\dfrac{\rho^0}{\rho_g}\right)\right)\,, \end{equation}
where $D=D(x)$ is given by
\begin{align*}
D &= 2\int_V \dfrac{v^2g^2dv}{Kg + \lambda}  - 2\left(\int_V \dfrac{vg\,dv}{Kg + \lambda}\right)^2 \left(\int_V \dfrac{dv}{Kg + \lambda}\right)^{-1} 
= 2Z\, {\rm Var}_{p}\left( vg \right)\, ,\\
& \mbox{with}\qquad p\,dv = \dfrac{dv}{Z(Kg + \lambda)} \,,\qquad Z = \int_V \frac{dv}{Kg+\lambda} \,. 
\end{align*}
\end{proposition}

\begin{remark}
Note that the macroscopic evolution operator and the simplification $-(\T\Pi)^* \T\Pi$ only differ by the diffusivities $D$ vs. $m_g$.
Under the assumption \eqref{g-ass1}, both have the same behaviour as $|x|\to\infty$.
\end{remark}

\begin{proof} 
With $\T\Pi f^0= \T f^0 = vg\,\partial_x (\rho^0/\rho_g)$ and Lemma \ref{lem-Lf=h}, we obtain
$$
  \J\T\Pi f^0 = \mu g + \frac{2g}{\lambda+gK} \left( \frac{\int_V vg/(\lambda+gK)dv}{\int_V 1/(\lambda+gK)dv} - vg\right) 
    \partial_x \left(\frac{\rho^0}{\rho_g} \right) \,.
$$
It remains to apply $-\Pi\T$ to this expression, where we can use \eqref{rhoTf}.
\end{proof}

\subsection*{Weakly biased turning rate}

We assume $0<\chi\ll 1$ and introduce the (parabolic) macroscopic rescaling $x\to \frac{x}{\chi}$, $t\to \frac{t}{\chi^2}$:
\begin{eqnarray*}
  &&\chi^2 \partial_t f + \chi\,v\,\partial_x f = Q_0(f) + \chi Q_1(f) \,,\\
  &&\mbox{with}\quad Q_0(f) = \int_V(f'-f)dv' \,,\quad Q_1(f) = \int_V (f'\sign(xv') - f\sign(xv))dv' \,.
\end{eqnarray*}
The standard Hilbert expansion $f = f_0 + \chi f_1 + O(\chi^2)$ (analogously to above) leads to 
$$
  f_0(t,x,v) = \rho_0(t,x) \,,\quad v\,\partial_x f_0 = Q_0(f_1) + Q_1(f_0) \,,\quad \partial_t \rho_0 + \partial_x \int_V vf_1dv = 0 \,.
$$
Multiplication of the second equation by $v$ and integration finally gives
$$
  \partial_t \rho_0 = \partial_x \left( \frac{1}{12}\partial_x\rho_0 + \frac{\sign(x)}{4}\rho_0\right) \,.
$$
The equilibrium distributions $\rho_g$ and $e^{-3|x|}$ of the two macroscopic limits share the exponential behavior for $|x|\to \infty$
up to the decay rate (by \eqref{g-ass1}), which is a consequence of the shared asymptotic behavior of the diffusivities 
$D/\rho_g$ vs. $1/12$ and drift velocities $-D\partial_x(1/\rho_g)$ vs. $-\sign(x)/4$.

\subsection*{Acknowledgements.}
Part of this work was performed within the framework of the LABEX MILYON (ANR-10-LABX-0070) of Universit\'e de Lyon, within the program "Investissements d'Avenir" (ANR-11-IDEX-0007) operated by the French National Research Agency (ANR).

\end{document}